\documentclass[reqno]{amsart}

\usepackage{amsthm, amsmath, amsfonts}
\usepackage{thmtools}
\usepackage{hyperref}
\usepackage{amssymb}
\usepackage{amscd}
\usepackage{mathrsfs}
\usepackage{geometry}
\usepackage{tikz-cd}
\usepackage{mathbbol}

\usepackage[nameinlink]{cleveref}

\theoremstyle{plain}
\newtheorem{theorem}{Theorem}[section]
\newtheorem{lemma}[theorem]{Lemma}

\newtheorem{definition}[theorem]{Definition}

\newtheorem{commentary}[theorem]{Human Comment}

\hypersetup{
    colorlinks=true,
    linkcolor=blue,
    citecolor=blue,
    urlcolor=blue
}

\newcommand{\Aletheia}{\emph{Aletheia}}
\renewcommand{\P}{\mathbb{P}}
\newcommand{\Z}{\mathbb{Z}}
\newcommand{\V}{\mathscr{V}}
\newcommand{\M}{\mathscr{M}}

\newcommand{\Hodge}{\Omega}

\newcommand{\C}{\mathbb{C}}

\DeclareMathOperator{\id}{id}

\DeclareMathOperator{\tr}{tr}

\DeclareMathOperator{\Mod}{Mod}

\usepackage[utf8]{inputenc}
\usepackage[T1]{fontenc}
\usepackage{xcolor} 
\usepackage[most]{tcolorbox} 

\definecolor{humanbubble}{RGB}{235, 245, 255}
\definecolor{aibubble}{RGB}{245, 245, 245}    
\definecolor{loggray}{RGB}{230, 230, 230}
\definecolor{indexbg}{RGB}{248, 248, 248}

\newtcolorbox{interactionlog}[2][]{
  enhanced,
  arc=0pt, outer arc=0pt,
  colback=white, colframe=black!60,
  boxrule=0.8pt,
  fonttitle=\bfseries\sffamily, coltitle=black, colbacktitle=loggray,
  title={Human-AI Interaction Card \if\relax\detokenize{#1}\relax\else for #1\fi},
  halign title=center, attach title to upper,
  after title={\vspace{4pt}\hrule\vspace{10pt}},
  lower separated=true,
  segmentation style={solid, black!60, line width=0.8pt},
  colbacklower=indexbg,
  after upper={\par\vfill
    \begin{tcolorbox}[
      enhanced, colback=indexbg, colframe=white, boxrule=0pt,
      top=0pt, bottom=0pt, fontupper=\footnotesize\sffamily,
      title=Raw prompts and outputs, coltitle=black!70, attach title to upper,
      after title={:\enskip}, sharp corners
    ]
    #2
    \end{tcolorbox}
  }
}

\newcommand{\human}[1]{%
  \noindent\begin{flushright}
    \begin{minipage}[c]{0.70\textwidth}
      \begin{tcolorbox}[
        enhanced,
        colback=humanbubble, colframe=black!15,
        arc=6pt, sharp corners=southeast, boxrule=0.5pt,
        left=6pt, right=6pt, top=4pt, bottom=4pt, boxsep=0pt
      ]\small #1\end{tcolorbox}
    \end{minipage}%
    \hspace{8pt}
    \begin{minipage}[c]{40pt}
      \footnotesize\sffamily\textbf{Human}
    \end{minipage}
  \end{flushright}
}

\newcommand{\ai}[2]{%
  \noindent\begin{flushleft}
    \begin{minipage}[c]{45pt}
      \footnotesize\sffamily\textbf{#1}
    \end{minipage}%
    \hspace{2pt}
    \begin{minipage}[c]{0.70\textwidth}
      \begin{tcolorbox}[
        enhanced,
        colback=aibubble, colframe=black!15,
        arc=6pt, sharp corners=southwest, boxrule=0.5pt,
        left=6pt, right=6pt, top=4pt, bottom=4pt, boxsep=0pt
      ]\small #2\end{tcolorbox}
    \end{minipage}
  \end{flushleft}
}

\title{The simplicity of the Hodge Bundle}
\author{Anand Patel}
\address{Department of Mathematics, Oklahoma State University, Stillwater, OK 74074}
\email{anand.patel@okstate.edu}

\begin{document}

\maketitle

\section{Introduction}

The Hodge bundle $\Omega_{g}$ is the rank $g$ vector bundle
over the moduli stack of complex genus $g$ curves $\M_{g}$  which assigns to a
curve $C \in \M_{g}$ its
vector space $H^{0}(C,\omega_{C})$ of holomorphic differentials.  \footnote{We will assume
$g \geq 2$ throughout.} In this note,
we prove that $\Omega_{g}$ is simple, meaning:

\begin{theorem}
  \label{theorem:main} The Hodge bundle $\Omega_{g}$ does not contain any
  non-trivial sub-bundles.
\end{theorem}

We will contextualize \Cref{theorem:main} after some
important acknowledgments.

\subsection{Acknowledgments and use of AI} The mathematical content in this paper came from
the output of \Aletheia, a custom AI agent (powered by Gemini Deep Think) built by Tony Feng \cite{feng2026towards}.  The author provided a
single prompt in which he asked for a proof of the stated theorem, and gave
no further hints. We provide the Human-AI interaction card (see \cite{feng2026towards}) below. As in
\cite{feng2026eigenweights},  the author made only expository adjustments, made
some proofs more readable, split others apart, and added some comments and questions
for further exploration.  In contrast, our subject matter here feels quite different and the prompt was not divided into multiple parts.  The author thanks Tony Feng for providing access to \Aletheia. 

It was Dawei Chen who first asked the author the
question about the simplicity of the Hodge bundle, after which point the author heard the
question several more times from other mathematicians circa 2015. Over the
years, the author has thought on and off about this question with friends and collaborators, particularly Anand Deopurkar, Eduard
Duryev, and Leonid Monin.  He thanks all of them for the fun conversations.  

The dual purpose of this note is to simultaneously address a curiosity among
members of the $\mathscr{M}_g$ community while also demonstrating the
capabilities (as of March 2026) of AI models when it comes to mathematical research.

\begin{interactionlog}[]{https://github.com/google-deepmind/superhuman/tree/main/aletheia}

    \human{Ask to prove simplicity of the Hodge bundle.}

    \ai{\Aletheia}{Fully correct solution exploiting curves with $(\mathbb{Z}/2\Z)^2$ symmetry group and a bit of  character theory.}

\end{interactionlog}

\subsection{Context and further directions}

\Cref{theorem:main} is another instance of the general 
philosophy that ``the universal object holds no surprises.'' It is important to note that the
analogous problem over the Deligne-Mumford compactification $\overline{\M}_{g}$,
where $\Omega_{g}$ is replaced by its usual extension (parameterizing sections of
the dualizing sheaf), is trivially easy:  One looks at a stable curve $C$ obtained by
attaching $g$ copies $E_{i}$ of the same elliptic curve $E$ to $\P^{1}$ at $g$ points $p_{1}, \dots, p_{g}$ on $\P^{1}$.
Then any sub-bundle must restrict to a direct sum of some set of
the $g$ natural lines $H^{0}(E_{i},\omega_{C}|_{E})$ in the space of differentials on this curve.  But these
lines are permuted transitively under monodromy when moving the $g$ points of
attachment, showing that the sub-bundle must have rank $g$.  This sort of argument also precludes the existence of sub-bundles of
the Hodge bundle over the space of principally polarized abelian varieties $\mathscr{A}_{g}$.  

The proof of \Cref{theorem:main} given here is similar in spirit
to the simple argument given above for $\overline{\M}_{g}$ and
$\mathscr{A}_{g}$ in that it exploits curves with non-trivial automorphisms.
The challenge lies in finding the exact constellation of curves to exploit and
how to exploit them.

There are more directions to go
beyond \Cref{theorem:main}.  For instance, if we let
$\Omega_{g}^{(n)}$ denote the vector bundle parameterizing sections of
$\omega_{C}^{\otimes n}$  (so that $\Omega_{g}^{(1)} = \Omega_{g}$) then one
could ask the same question as to whether any holomorphic sub-bundles exist.
This expanded setting brings with it some complications: For infinitely many values of
$n$ there do exist holomorphic line sub-bundles of $\Omega_{g}^{(n)}$ due to the existence of various types
of Weierstrass divisors. It is then natural to wonder, for instance, if Weierstrass divisors
account for all line sub-bundles. This would be a rigidity result much in the spirit of those discussed in   \cite{farb2023rigidity}.

\subsection{Conventions} We will work over the complex numbers $\C$. \footnote{Though after
making adjustments as in  \Cref{commentary:hurwitzspaces} one
should be able to prove the result in all characteristics outside $2$ almost
immediately from the strategy given. \Cref{theorem:main} should be true over any field.} The term \textit{curve} assumes 
irreducible, smooth, and projective by default, and if $C$ is a curve then
$\omega_{C}$ denotes its canonical line bundle.  If $G$ is a finite group acting
on a curve $C$ then $C/G$ denotes the classical quotient curve rather than the
quotient stack.

\section{The proof} 

From this point onward, we posit the existence of a sub-bundle which we denote by \[\V
\subset \Hodge_{g},\] and we let $r$ denote the rank of $\V$.  The first
observation we make is that if $C$ is any genus $g$ curve, and if $G$ is any group of
automorphisms of $C$, then the fiber \[V_{C} \subset H^{0}(C,\omega_{C})\] of $\V$ at $C \in \M_{g}$ is a
sub-representation of the $G$-representation $H^{0}(C,\omega_{C})$ given by
pullback of forms.

We will frequently refer to the set \[S := \{m \mid 0 \leq m \leq
g+1\,\,\textrm{and} \,\, m \equiv g+1 \pmod 2 \}.\]
The significance of the set $S$ is: a genus $g$ curve has an involution with
$2m$
fixed points if and only if $m \in S$.  

\subsection{Curves with involutions}

\begin{lemma}
  \label{corollary:independent} Let $\tau$ be an involution on a genus $g$ curve
  $C$ with $2m$ fixed points.  Then the trace
  of the restriction of $\tau^{*}$ on $V_{C}$ depends only on the integer $m$ and the bundle
  $\V$. 
\end{lemma}

\begin{proof}
By the topological classification of surfaces, two orientation-preserving
involutions on a closed orientable surface $\Sigma_g$ with the same number of
fixed points correspond to double covers over a quotient surface of a uniquely
determined genus.  Thus, such involutions are conjugate in the mapping class
group $\Mod_{g}$.  The putative bundle $\V$ corresponds to a
$\Mod_g$-equivariant bundle over the Teichm\" uller space $\mathscr{T}_{g}$,
which we continue to denote by $\V$.  

By the Nielson Realization Theorem \cite{kerckhoff1983nielsen} the fixed point locus
$\mathscr{T}_{g}^{\tau}$ of $\tau$ is non-empty and naturally identifies with
the quotient orbifold $\Sigma_g / \langle\tau \rangle$.  Since the Teichm\"
uller space of any surface orbifold is homeomorphic to a Euclidean space,
$\mathscr{T}_{g}^\tau$ is a connected manifold.  The restriction of $\V$ to
$\mathscr{T}_{g}^{\tau}$ globally splits into $+1$ and $-1$ eigen-bundles for
the action of $\tau^{*}$, each having a constant rank by connectedness.  The trace of $\tau^{*}$ on $V_C$ is then
the difference between the ranks of these bundles, proving the lemma.

\end{proof} 

\begin{commentary}
  \label{commentary:hurwitzspaces} Aletheia decided to go the topological route
  with \Cref{corollary:independent}, presumably because the author
phrased the initial question over the complex numbers.  But this argument can be
adapted to be purely algebraic by appealing to the irreducibility of the Hurwitz
stack representing double covers $C \to D$ where $C$ is a genus $g$ curve and
where the double cover has $2m$ branch points. Such an approach would also avoid the use of
the rather heavy Nielson Realization Theorem. Irreducibility of these stacks is proven
using standard Hurwitz monodromy arguments, and is valid over any field of
characteristic other than $2$.  

One advantage of the topological approach: it shows that $\Omega_{g}$ does
not contain any \textit{continuous} sub-bundles either.  This suggests there might
also be cohomological arguments for \Cref{theorem:main}. 
\end{commentary}

\Cref{corollary:independent} paves the way for the next
definition.

\begin{definition}
  \label{def:functionf} For each $m \in S$, we let $f(m)$ denote the trace
  $\tau^{*}$ on $V_{C}$ for any curve $C$ admitting an involution $\tau$ with
  $m$ fixed points.
\end{definition}

\subsection{Some curves with two commuting involutions} The curves we will consider are certain
curves with a $(\Z/2\Z)^{2}$ symmetry group whose quotient yields $\P^{1}$.  Let $C$ be
such a curve. Then it has three commuting involutions $\tau_1, \tau_2, \tau_3$
satisfying $\tau_1 \tau_2 = \tau_3$; 
let $2m_1, 2m_2, 2m_3$ denote the corresponding numbers of fixed points. We
assume these fixed point sets are pair-wise disjoint. 
 
We will first determine which triples of integers $(m_1, m_2, m_3) \in S^{3}$ can arise in
this way. We let \[T \subset S\times S \times S\] denote the set of such triples.  The answer is straightforward.   

\begin{lemma}
  \label{lemma:triples} A triple $(m_1, m_2, m_3)$ lies in $T$ if and only if at
  most one of the $m_i$ is $0$ and $m_1 + m_2 + m_3 = g+3$. 
\end{lemma}

\begin{proof}
  Given a triple $(m_1,m_2,m_3) \in T$, we choose generic, mutually coprime,
  square-free polynomials $A,B,C \in \C[x]$ of degrees $m_1, m_2, m_3$,
  respectively. Let $K = \C(x)$.  We define a smooth projective curve $C$ as the
  smooth model of the function field $L = K(y,z)$ where $y^{2} = AC$ and $z^{2}
  = BC$.  Because at most one $m_i = 0$, the polynomials $AC$ and $BC$ cannot
  both be constants; thus the extension $L/K$ forms a fully generated degree $4$
  Galois extension with Galois group $G \cong (\Z/2\Z)^{2}$, guaranteeing $C$ is
  a connected curve.  Since $m_i \equiv g+1 \pmod 2$, the polynomials $AC$ and
  $BC$ have even degrees, implying that $L/K$ is unramified over $x = \infty \in
  \P^{1}$. 

  By construction, the ramification of the cover $C \to \P^{1}$ lies entirely
  over the roots of $A,B$ and $C$.  Over a root $x=p$ of $A$ or $B$, the ramification
  profile is of type $(2,2)$ -- that is, there are two points lying over $p$
  with simple ramification occurring at each point. 

  Over a root $x=c$ of $C(x)$, the quadratic subextension $K(y/z)/K$ evaluates locally to
  $(y/z)^{2} = A(c)/B(c) \neq 0$ and therefore by adjoining $y$ or $z$ we again
  obtain two points of simple ramification over $q$.  

  By the Riemann-Hurwitz formula applied to $L/K$, the genus $g_C$ satisfies: 
  \[2g_{C}-2 = 4(-2) + 2m_1 + 2m_2 + 2 m_3 = -8 + 2(g+3) = 2g-2,\] which
  confirms $g_{C}=g$. The automorphisms $\tau_1, \tau_2$ acting by $(y,z) \mapsto (-y,z)$
 and $(y,z) \mapsto (y,-z)$, respectively, generate the Galois group $G$. Because $L
 = K(z)(y)$ and $y^{2} = AC$, ramification in $C \to C_{2}$ occurs exactly where
 the local valuation of $AC$ in $K(z)$ is odd.  Over a root $x=a$ of $A$, $AC$ has a
 simple zero, and since $C_{2} \to \P^{1}$ is unramified there the valuation
 remains $1$ at the two points lying in $C_{2}$ over $a$.  This gives
 $2m_1$ fixed points so far for $\tau_1$.  

 At the unique point $t \in C_{2}$ lying over a root $x = c$ of $C(x)$, the local function $z$ is a local coordinate. 
 Thus $AC$ vanishes to order $2$ at $t$. Over $t$ the double cover $C \to C_2$ is unramified, and so $\tau_{1}$ swaps the two 
 resulting pre-image points of $t$ in $C$.   Therefore, the fixed point set of $\tau_1$ consists of $2m_1$ points exactly. 

 Similar analyses for $\tau_2$ and $\tau_3 = \tau_1 \tau_{2}$, yield the conclusion that the number of fixed points
 of $\tau_i$ is $2m_i$.

 Conversely, if $C$ is a genus $g$ curve with commuting involutions $\tau_1, \tau_2,
 \tau_3 = \tau_1 \tau_2$ with disjoint fixed point sets of size $2m_{i}$
 respectively, and such that the quotient $C/\langle \tau_1, \tau_2 \rangle
 \simeq \P^{1}$, then the Riemann-Hurwitz analysis performed earlier for the
 cover $C \to \P^{1}$ shows that
 $m_1 + m_2 + m_3 = g+3$ holds. Furthermore, each $m_{i} \equiv g+1 \pmod 2$ by
 applying Riemann-Hurwitz to the quotients $C \to C /\tau_i$. Two of the $m_i$, say $m_1, m_2$ cannot simultaneously be $0$ since then $C \to
 \P^{1}$ would be a degree $4$ connected unramified
 cover of $\P^{1}$, which is impossible. The lemma is shown. 
\end{proof}

We now fix any curve $C$ which has a group of automorphisms $G \simeq
(\Z/2Z)^{2}$ with quotient $C/G \simeq \P^{1}$ as above. We
let $V^{++}, V^{+-}, V^{-+}, V^{--}$ denote the sub-representations of $V_{C}$ on
which $\tau_{1},\tau_{2}$ act via $+1$ or $-1$ as indicated in the superscripts. The subspace $V^{++}$ is the space  
of $G$-invariant differentials, which are exactly the pullbacks of differentials on $\P^{1}$. Since
$\P^{1}$ has no non-zero global $1$-forms, it follows that $V^{++} = 0$.

We'll need a simple fact about finite group representations.
\begin{lemma}
  \label{lemma:character} Let $G$ be a finite group acting on a
  finite-dimensional vector space $V$, and let $\chi$ denote a character of $G$.
  Let $V^{\chi}$ denote the sub-representation of $V$ on which $G$ acts via
 scaling by $\chi$.  Then 
 \[\dim V^{\chi} = \frac{1}{|G|}\sum_{g}\chi(g)^{-1}\tr(g).\]
\end{lemma}

\begin{proof}
  Consider the linear operator $F: V \to V$ given by $\frac{1}{|G|}\sum_{g \in
  G}\chi(g)^{-1}g$. Then it follows immediately that $F(v) = v$ for any $v \in V^{\chi}$,
   and also that for any $w \in V$ and $h \in G$, \[h(F(w)) = \chi(h)w.\] Therefore, $F$ is a projector with image $V^{\chi}$. 
  Since the trace of a projector is the dimension of its image, we conclude that \[\dim V^{\chi} = \frac{1}{|G|}\sum_{g \in G}\chi(g)^{-1}\tr{g}\] as claimed.
\end{proof}

\begin{commentary}
  \label{commentary:finiterep} The human author added this proof of the rather
  standard \Cref{lemma:character} for sake of completeness.  It
  seems that Aletheia wrongly assumed this human would just know it. 
\end{commentary}

\begin{lemma}
  \label{lemma:negativer} For any $(m_1,m_2,m_3) \in T$, we have 
  \[r + f(m_1) + f(m_2) + f(m_3) = 0,\] and 
  \[r + f(m_1) - f(m_2) - f(m_3) = 4\dim V^{+-}.\]
\end{lemma}

\begin{commentary}
  \label{commentary:droppedfour} In a previous version of this article, the human author, in the process of re-writing the raw output, accidentally dropped the $4$ in front of $\dim V^{+-}$. The rest of the proof is not affected.
\end{commentary}

\begin{proof}
  We apply \Cref{lemma:character} to the identity character,
  in combination with the fact that $V^{++} = 0$ (no invariant differentials)
  and \Cref{def:functionf} to conclude both equations.
\end{proof}

\begin{lemma}
  \label{lemma:f(m)} For each $m \in S$ the integer $f(m)$ (see
  \Cref{def:functionf}) is equal to $\frac{r(1-m)}{g}.$
\end{lemma}

\begin{proof}
  Let $y(m) = f(m) - \frac{r(1-m)}{g}$; then by definition of the set $T$ and by
  \Cref{lemma:negativer} it follows that 

  \begin{equation}
    \label{eq:sumy}
    y(m_1) + y(m_2) + y(m_3)
  = 0
\end{equation} for any $(m_1, m_2, m_3) \in T$.  Clearly, it suffices to show that $y(m) = 0$ for all $m \in S$. 

  In the most extreme case that $m = g+1$, any involution with $2m = 2g+2$ fixed points must be
  hyperelliptic, and therefore acts by $-\id$ on the space of differentials.
  Hence, $f(g+1) = -r$ and so \[y(g+1) = 0.\]

  Now there are two slightly different cases depending on the parity of $g$.  

  \textbf{If $g$ is even}: Then the set $S$ consists of the odd integers
  $2j+1$ for all $0 \leq j \leq g/2$.  The element $(1, g+1, 1) \in
  T$ yields the relation \[2y(1) + y(g+1) = 0,\] and since we already know $y(g+1)=0$, we conclude $y(1) = 0$. 
  Next, by considering the two elements $(1,2j+3,g-2j-1)$ and $(3,2j+1,g-2j-1)$
  in $T$, we conclude by subtracting corresponding relations \eqref{eq:sumy}
  that the integers $y(2j+1)$ form an arithmetic progression.  But since $y(1) =
  y(g+1) = 0$, we conclude $y \equiv 0$ as claimed. 

  \textbf{If $g$ is odd.} The set $S$ consists of the even integers $2k$ for $0
  \leq k \leq (g+1)/2$.  By considering the elements $(0,2k+2,g+1-2k)$ and
  $(2,2k,g+1-2k)$ of $T$, we again conclude that the integers $y(2k)$ are in
  arithmetic progression, satisfying \[y(2k+2) - y(2k) = y(2) - y(0).\]  By
  considering the relation \eqref{eq:sumy} (in conjunction with
  the fact that $y(g+1) = 0$)  for the triple
  $(0,2,g+1) \in T$, we get $y(0) + y(2) = 0$. Combining yields \[y(2k) =
  y(0)(1-2k)\] for all $0 \leq k \leq (g+1)/2$. Since $g \geq 3$, by
  plugging in $k = (g+1)/2$ we conclude that $y(0)=0$, and so $y(2k)=0$ for all $k$ as
  promised.  
\end{proof}

\subsection{Proof}

\begin{proof} [Proof of \Cref{theorem:main}]
  \label{proof:main} 

  It suffices to prove that the rank $r$ of the sub-bundle $\V \subset
  \Hodge_{g}$ satisfies the congruence $r \equiv 0 \pmod g$. 

  Again we let $C$ be a curve with two commuting involutions
  $\tau_1, \tau_2$, with $\tau_3 := \tau_1 \tau_2$ and with $2m_1, 2m_2, 2m_3$ fixed
  points respectively.  Combining both equations of
  \Cref{lemma:negativer}, we obtain 
  \[\dim V^{+-} = \frac{1}{2}(r + f(m_1))\] and so $r+f(m_1)$ is an even
  integer for any $m_1 \in S$ which can be completed to a triple in $T$. \footnote{This is where \Cref{commentary:droppedfour} is important.}  By
  \Cref{lemma:f(m)} \textit{we conclude that $r + \frac{r(1-m_1)}{g}$ must
  be an even integer for all such $m_{1} \in S$.}

  \textbf{If $g$ is even}: Then both $1$ and $3$ belong to $S$ and can be
  extended to triples $(1,1,g+1)$ and $(3,1,g-1)$ in $T$. Considering these, we
  conclude that $r$ must be even and that $r- \frac{2r}{g}$ must also be even.
  Subtracting gives the conclusion that $r \equiv 0 \pmod g$. 

  \textbf{If $g$ is odd.} Then both $0$ and $2$ belong in $S$ and extend to
  triples $(0,2,g+1)$ and $(2,2,g-1)$ in $T$. Therefore, the italicized
  observation above yields that both $r + \frac{r}{g}$ and $r -\frac{r}{g}$ are
  even integers.  As before, subtracting implies $r \equiv 0 \pmod g$, concluding
  the proof of \Cref{theorem:main}.
\end{proof}

\bibliographystyle{alpha}
\bibliography{simplehodge}

@article{feng2026eigenweights,
  title = {{E}igenweights for arithmetic {H}irzebruch {P}roportionality},
  author = {Feng, Tony},
  journal = {arXiv preprint arXiv:2601.23245},
  year = {2026},
}

@article{kerckhoff1983nielsen,
  title = {The {N}ielsen realization problem},
  author = {Kerckhoff, Steven P},
  journal = {Annals of mathematics},
  volume = {117},
  number = {2},
  pages = {235--265},
  year = {1983},
  publisher = {JSTOR},
}

@article{feng2026towards,
  title = {Towards autonomous mathematics research},
  author = {Feng, Tony and Trinh, Trieu H and Bingham, Garrett and Hwang, Dawsen
            and Chervonyi, Yuri and Jung, Junehyuk and Lee, Joonkyung and Pagano,
            Carlo and Kim, Sang-hyun and Pasqualotto, Federico and others},
  journal = {arXiv preprint arXiv:2602.10177},
  year = {2026},
}

@article{farb2023rigidity,
  title={Rigidity of moduli spaces and algebro-geometric constructions},
  author={Farb, Benson},
  journal={arXiv preprint arXiv:2302.06369},
  year={2023}
}
\end{document}